\documentclass[12pt,a4paper]{amsart}
\usepackage{graphicx,amssymb}
\usepackage{amsmath,amssymb,amscd}
\input xy
\xyoption{all}
\usepackage[all]{xy}
\usepackage{hyperref}
\usepackage{tikz}
\usetikzlibrary{arrows,shapes,trees}

\newcommand{\ra}{\rightarrow}

\newcommand{\CC}{\mathbb C}

\newcommand{\cO}{\mathcal{O}}

\newcommand{\Ext}{\mbox{Ext}}

\newcommand{\Hom}{\mbox{Hom}}

\newcommand{\Cliff}{\mbox{Cliff}}
\newcommand{\gr}{\mbox{gr}}

\theoremstyle{plain}
\newtheorem{theorem}{Theorem}[section]
\newtheorem{lem}[theorem]{Lemma}
\newtheorem{prop}[theorem]{Proposition}
\newtheorem{cor}[theorem]{Corollary}

\newtheorem{rem}[theorem]{Remark}
\newtheorem{ex}[theorem]{Example}
\numberwithin{equation}{section}
\begin{document}
\title[BN-theory for genus 4]{Higher rank BN-theory for curves of genus 4}

\author{H. Lange}
\author{P. E. Newstead}

\address{H. Lange\\Department Mathematik\\
              Universit\"at Erlangen-N\"urnberg\\
              Cauerstra\ss e 11\\
              D-$91058$ Erlangen\\
              Germany}
              \email{lange@mi.uni-erlangen.de}
\address{P.E. Newstead\\Department of Mathematical Sciences\\
              University of Liverpool\\
              Peach Street, Liverpool L69 7ZL, UK}
\email{newstead@liv.ac.uk}

\date{\today}
\thanks{Both authors are members of the research group VBAC (Vector Bundles on Algebraic Curves). The second author 
would like to thank the Department Mathematik der Universit\"at 
        Erlangen-N\"urnberg for its hospitality}

\begin{abstract}
Higher rank Brill-Noether theory is completely known for curves of genus $\leq 3$.
In this paper, we investigate the theory for curves of genus 4. Some of our results apply to curves of arbitrary genus.
\end{abstract}
\maketitle

\section{Introduction}\label{intro}

Let $C$ be a smooth complex projective curve and let $B(n,d,k)$ denote the {\it Brill-Noether locus} of stable bundles on $C$ of rank 
$n$ and degree $d$ with at least $k$ independent sections (for the formal definition, see Section \ref{back}). This locus has a natural 
structure as a subscheme of the moduli space of stable bundles on $C$ of rank $n$ and degree $d$.

In the case $n=1$, the Brill-Noether loci are classical objects. The theory of such loci was given 
firm foundations in the 1980s by a number of authors (see \cite{acgh} for further details). For $n>1$, 
the study began towards the end of the 1980s and the situation is much less clear, even on a general 
curve. Although a great deal is known about non-emptiness of Brill-Noether loci (see for example 
\cite{te,bgn,m1,bmno} and many other papers), there is much that is not known. The problem is completely 
solved only for $g\le3$ (see Section \ref{back} for details), although there are strong results for hyperelliptic and bielliptic curves (see \cite{bmno} and \cite{ba}). 

Our primary object in this paper is to investigate the case of non-hyperelliptic (hence trigonal) 
curves of genus $4$. The main result of the paper (Theorem \ref{thm4.8}) concerns new upper 
bounds on $k$ for the non-emptiness of $B(n,d,k)$ and the corresponding loci $\widetilde{B}(n,d,k)$ for 
semistable bundles, which improve the presently known bounds (for which, see Proposition \ref{p1.5}) and 
are close to being best possible in the case when $C$ has two distinct trigonal bundles. We also produce 
a large number of examples of stable bundles which come close to attaining these upper bounds.
Many of these are constructed using elementary transformations, the only problem here being to prove stability.
Some of these are known, but many are new.

In Section \ref{back}, we give some background and describe some known results, which are sufficient to give 
a complete answer to the non-emptiness problem for $g\leq3$ and for hyperelliptic curves of genus $4$. In Section \ref{nonh}, 
we obtain upper bounds for non-hyperelliptic curves 
of arbitrary genus and also construct some examples for such curves. 
Section \ref{upper} contains further results on upper bounds for non-hyperelliptic curves of genus $4$, 
leading to our main result Theorem \ref{thm4.8}. In Section \ref{exist}, we construct many examples of 
stable bundles with sections.  In Section 6 we consider extremal bundles (i.e. those giving equality in our upper bounds) 
and describe all such bundles. We consider also bundles of low rank ($\leq 4$).
Finally, in Section \ref{map}, we provide a graphical representation of our results.

Our methods are inspired in particular by those of \cite{bmno} and work of Mercat \cite{m1,m2}. In addition, 
a variant of the statement of Lemma \ref{l2.4},  both statement and proof of Lemma \ref{l2.8} and the statement (for 
arbitrary $\delta$ but without proof) of Example \ref{ex5.2} are contained in unpublished work of Mercat.

\section{Background and some known results}\label{back}

Let $C$ be a smooth complex projective curve. Denote by $M(n,d)$ the moduli space of stable vector 
bundles of rank $n$ and degree $d$ and by $\widetilde M(n,d)$ the moduli space of 
S-equivalence classes of semistable bundles of rank $n$ and degree $d$. For any integer
$k \geq 1$, we define 
$$
B(n,d,k) := \{ E \in M(n,d) \; | \; h^0(E) \geq k \}
$$
and
$$
\widetilde B(n,d,k) := \{ [E] \in \widetilde M(n,d)  \;|\; h^0( \gr E) \geq k \},
$$
where $[E]$ denotes the S-equivalence class of $E$ and $\gr E$ is the graded object defined 
by a Jordan-H\"older filtration of $E$. The locus $B(n,d,k)$ has an {\it expected dimension}
\begin{equation} \label{equ2.1}
\beta(n,d,k):=n^2(g-1)+1-k(k-d+n(g-1)),
\end{equation}
known as the {\it Brill-Noether number}.
For any vector bundle $E$ on $C$, we write $n_E$ for the rank of $E$, $d_E$ for the degree of $E$ and $\mu(E)=\frac{d_E}{n_E}$ for the slope of $E$. The vector bundle $E$ is said to be {\it generated} if the evaluation map $H^0(E)\otimes {\mathcal O_C}\to E$ is surjective. 

An important method of constructing stable bundles is a special case of the {\it dual span} construction
defined as follows. Let $L$ be a generated line bundle on $C$ with $h^0(L) \geq 2$.  Consider
the evaluation sequence
$$
0 \ra E_L^* \ra H^0(L) \otimes \cO_C \ra L \ra 0.
$$
Then $E_L$ is a bundle of rank $h^0(L) -1$ and degree $d_L$ with $h^0(E_L) \geq h^0(L)$.
 It is called the {\it dual span} of $L$ and is also denoted by $D(L)$. Although $E_L$ is not 
necessarily stable, this is frequently the case. The use of the dual span has a long history, but in 
our context especially important are \cite{bu} and \cite{m1}. The construction can of course be carried out for a generated vector bundle in place of the line bundle $L$, but we shall not need this.

We are interested in investigating the non-emptiness of $B(n,d,k)$ and $\widetilde B(n,d,k)$. The results can be illustrated on the {\it Brill-Noether map} (BN-map) in which $\lambda=\frac{k}n$ is plotted against $\mu=\frac{d}n$ (see Section \ref{map} for the case of genus $4$). On this map, the {\it Brill-Noether curve} (BN-curve) is of particular significance. It is a portion of a hyperbola given by $\lambda(\lambda-\mu+g-1)=g-1$, which is equivalent to $\beta(n,d,k)=1$. So we should expect points on the BN-map to be located below or possibly just above this curve. This is by no means always the case, but it does give a good indication of where to look for non-empty Brill-Noether loci.

By Serre duality, it is sufficient 
to consider the case $d \leq n(g-1)$ and we always have $d \geq 0$. For $g = 0$ and $g=1$,
there is nothing to be done. For $g=2$, a complete answer is contained in \cite{bgn}. The next 
two known propositions completely cover the case $g=3$ and also give information for higher 
values of $g$.

\begin{prop} \label{p1.1}
Let $C$ be a curve of genus $g\geq 3$ and $0 \leq d < 2n,\; k \geq 1$.Then $B(n,d,k) \neq 
\emptyset$ if and only if one of the following holds.
\begin{enumerate}
 \item[(i)] $(n,d,k) = (1,0,1)$; 
 \item[(ii)] $(n,d,k) = (1,1,1)$;
 \item[(iii)] $d>0,\; k-n  \leq \frac{1}{g}(d-n)$ and $(n,d,k) \neq (n,n,n)$.
\end{enumerate}
Moreover, $\widetilde B(n,d,k) \neq \emptyset$ if and only if either $d=0,\;  k \leq n$ or $d > 0, \;k - n \leq \frac{1}{g}(d-n)$.
\end{prop}

This is contained in \cite[Theorems $B$ and $\widetilde B$]{bgn} and \cite{m1}.

\begin{prop} \label{p1.2}
 Let $C$ be a curve of genus $g \geq 3$ and $d = 2n, \; k \geq 1$.
 Then $B(n,d,k) \neq \emptyset$ if and only if one of the following holds.
 \begin{enumerate}
  \item[(i)] $C$ non-hyperelliptic and either $k \leq \frac{n(g+1)}{g}$ or $(n,d,k) = (g-1,2g-2,g)$; moreover, $B(g-1,2g-2,g) =  \{D(K_C)\}$;
  \item[(ii)] $C$ hyperelliptic and either $k \leq n$ or $(n,d,k) = (1,2,2)$.
 \end{enumerate}
 Moreover, $\widetilde B(n,d,k) \neq \emptyset$ if and only if either $C$ is non-hyperelliptic and $k \leq \frac{ng}{g-1}$ or $C$ is hyperelliptic and $k \leq 2n$.
\end{prop}

\begin{proof}
For the stable case, see \cite[Theorems 1 and 2]{m2}. The semistable case is easily deducible.
\end{proof}

The next proposition completely covers the case of hyperelliptic curves of genus 4 and gives information in higher genus.

\begin{prop} \label{p1.3}
Let $C$ be a hyperelliptic curve of genus $g \geq4$ and $2n < d < 4n$. Then $\widetilde B(n,d,k) \neq \emptyset$ if and only if 
one of the following holds.
\begin{enumerate}
\item[(i)] $d=3n$ and $k \leq 2n$;
\item[(ii)] $2n<d<3n, d = 3n-g\ell - \ell'$ with $2 \leq \ell' \leq g$ and $k \leq 2n - 2\ell -2$ or $\ell' = 1$ and $k \leq 2n-2\ell -1$;
\item[(iii)] $3n<d<4n, d = 3n+g \ell +\ell'$ with $0 \leq \ell' \leq g-2$ and $k \leq 2n + 2\ell$ or $\ell' = g-1$ and $k \leq 2n +2 \ell + 1$.
\end{enumerate}
Moreover, $B(n,d,k) \neq \emptyset$ if and only if $\widetilde B(n,d,k) \neq \emptyset$ and $(n,d,k) \neq (n,3n,2n)$ with $n \geq 2$.
\end{prop}

\begin{proof}
If $B(n,d,k) \neq \emptyset$, then $k \leq 2n + \frac{2}{g}(d-3n)$ by \cite[Theorem 6.2 (1)]{bmno}. This follows also for $\widetilde B(n,d,k)$ by considering a 
Jordan-H\"older filtration.

It is clear that $\widetilde B(n,3n,2n) \neq \emptyset$ (take a direct sum of suitable line bundles). The fact that, if $n \geq 2$, $B(n,3n,2n) = \emptyset$, but 
$B(n,3n,2n-1) \neq \emptyset$ is covered by \cite[Corollary 6.1 and Proposition 6.1]{bmno}. 
For (iii) in the stable case, see \cite[Example 6.1 and Remark 6.2]{bmno}. 
The semistable case follows easily. 

For (ii), one can argue as in \cite[Example 6.1]{bmno} for $2\le \ell'\le g$. For $\ell'=1$, first use (iii) to show that there exists a stable bundle $E$ of rank $n$ and degree $d$ with $h^0(E^*\otimes H^3)=2n+2\ell$, where $H$ is the hyperelliptic line bundle. Then repeated use of the formula $2h^0(F)\le h^0(F\otimes H^*)+h^0(F\otimes H)$ for any bundle $F$ gives $h^0(E^*\otimes K_C)=(g-2)n+(g-2)\ell$. The result follows by Serre duality and Riemann-Roch.
\end{proof}

\section{Non-hyperelliptic curves}\label{nonh}

We begin with a result which is well known.

\begin{prop} \label{p1.5}
Let $E$ be a semistable bundle on a non-hyperelliptic curve $C$ of rank $n \geq 2$ and
degree $d$ with 
$ 1 \leq \mu(E) \leq 2g-3$. Then 
$$
h^0(E) \leq \frac{1}{2}(d + n).
$$
\end{prop}

\begin{proof}
For the stable case, see \cite[Propositions 3 and 4]{re}; the semistable case is easily deducible by considering a Jordan-H\"older filtration.
\end{proof}

\begin{lem} \label{l3.1}
Let $C$ be a non-hyperelliptic curve and $E$ a semistable bundle of rank $n$ and degree $d$ with $\mu(E) \geq 3$. Then
$$
h^0(E) \leq d-n.
$$
\end{lem}

\begin{proof}
If $3 \leq \mu(E) \leq 2g-3$, then, by Proposition \ref{p1.5},
$$
h^0(E) \leq \frac{1}{2}(d+n) \leq d-n.
$$
If $\mu(E) > 2g-3$, then, by Riemann-Roch,
$$
h^0(E) = h^1(E) + d -n(g-1) \leq n+d - n(g-1) \leq d-n,
$$
since $K_C \otimes E^*$ is a semistable bundle of slope $< 1$ and we can apply Proposition \ref{p1.1}. 
 \end{proof}

\begin{prop} \label{p3.3}
Let $C$ be a non-hyperelliptic curve of genus $g \geq 4$. Then $B(n,d,k) \neq \emptyset$ for infinitely many $\mu$ with $2 < \mu < 3$ and $k > \frac{d}{2}$. 
\end{prop}

\begin{proof} Suppose $r \leq g-1$ and $s \geq 0$. Consider extensions of type
\begin{equation} \label{e4.4}
0 \ra F \ra E \ra \CC_p \ra 0
\end{equation}
with 
$$
F = D(K_C)^{\oplus r} \oplus M_1 \oplus \cdots \oplus M_s,
$$ 
where $d_{M_i} = 2, h^0(M_i) = 1$ for all $i$ and $M_i \not \simeq M_j$ for $i \neq j$.
Note that $D(K_C)$ has rank $g-1$ and degree $2g-2, h^0(D(K_C)) = g$ and $D(K_C)$ is stable by \cite[Corollary 3.5]{pr}. Moreover,
$$
n_E= r(g-1)+s, \quad d_E = 2r(g-1) +2s + 1 \quad \mbox{and} \quad h^0(E) \geq rg+s.
$$
Furthermore, $\dim \Ext^1(\CC_p, D(K_C)) = g-1$. Suppose that \eqref{e4.4} is classified by $(e_1, \dots, e_r,f_1, \dots,f_s)$ where $e_1, \dots,e_r$ are linearly independent
elements of $\Ext^1(\CC_p, D(K_C))$ and $f_i \in \Ext^1(\CC_p,M_i)$ with $f_i \neq 0$ for all $i$.

Suppose now that $G$ is a subbundle of $E$ contradicting stability.
Then $d_G > 2 n_G$. This implies that $d_{F \cap G} = 2n_G$ and there exists an extension
\begin{equation} \label{e3.2}
0 \ra F \cap G \ra G \ra \CC_p \ra 0 .
\end{equation}
Moreover, $F \cap G$ must be a direct factor of $F$, since all subbundles of $F$ contradicting stability are of this form. Now \eqref{e4.4} is induced from \eqref{e3.2}
via the inclusion $F \cap G \ra F$ which contradicts the choice of $(e_1, \dots,e_r,f_1,\dots,f_s)$. So $E$ is stable.

Thus we obtain $B(n,d,k) \neq \emptyset$ for 
\begin{equation} \label{eq3.3}
(n,d,k) = (r(g-1) + s, 2r(g-1) + 2s + 1, rg + s),
\end{equation}
whenever $r\le g-1$ and $s\ge0$.
\end{proof}

\begin{rem}
{\em The case $s=1$ of the construction of Proposition \ref{p3.3} is already known to yield interesting examples of stable bundles (see \cite[Proposition 7.2]{top3}).}
\end{rem}

Our next proposition is an extension of Proposition \ref{p3.3}. We begin with a lemma.

\begin{lem} \label{lem}
Suppose that $M_1, \dots, M_s$ are line bundles on a non-hyper-
elliptic curve $C$ with $d_{M_i} =2, \; h^0(M_i) = 1$. Write $M_i = \cO_C(q_{i1} + q_{i2})$ and suppose that 
if $i \neq j$, then $q_{ik} \neq q_{jl}$ for any $k,l$.
Assume that the $M_i$ do not admit a non-zero map into any trigonal line bundle on $C$. Suppose that $F$ is a bundle such that $d_F\ge2n_F-1$ and every quotient line bundle of $F$ has degree $\geq 1$. Suppose further that any subbundle $F'$ of $F$ with $d_{F'}\ge2n_{F'}-1$ also has the property that every quotient line bundle has degree $\ge1$.  Then $F$ admits 
non-zero homomorphisms into at most $n_F$ of the $M_i$.
\end{lem}

\begin{proof}
Consider first the case $n_F = 1$ and suppose that there exist non-zero homomorphisms of $F$ to $M_{i}$ and $M_{j}$ with $i \neq j$. If $d_F \geq 2$, this is impossible. 
If $d_F = 1$, we have $F \simeq M_{i}(-p_i) \simeq M_j(-p_j)$ for some points $p_i, p_j$. Then 
$$
\cO_C(q_{i1}+q_{i2}+p_j) \simeq \cO_C(q_{j1}+q_{j2}+p_i),
$$
contradicting the assumption on the $M_i$.

Now suppose that $n_F \geq 2$ and argue by induction. If there exists a non-zero homomorphism $F \ra M_j$, then the kernel $F'$ of the homomorphism has 
$n_{F'} = n_F -1$ and $d_{F'} \geq 2n_{F'} -1$. Applying the inductive hypothesis to $F'$ and noting that $F/F'$ is a line bundle of degree $\ge1$, we see that the result follows.
\end{proof}

\begin{prop} \label{p5.14}
 Let $C$ be a non-hyperelliptic curve and $M_1, \dots, M_s$ as in Lemma \ref{lem}. Let 
 \begin{equation} \label{e5.12}
 0 \ra D(K_C)^r \oplus M_1 \oplus \dots \oplus M_s \ra E \ra \CC_{p_1} \oplus \CC_{p_2} \ra 0 
 \end{equation}
 be an exact sequence with $p_1 \neq p_2$ and $r \leq g-1$. Suppose further that \eqref{e5.12} is classified by 
 $(e_{11}, \dots, e_{1r},e_{21}, \dots, e_{2r},f_{11},f_{12}, \dots,f_{s1},f_{s2})$ with $e_{ij}$ linearly independent elements of $\Ext^1(\CC_{p_i},D(K_C))$
 and $f_{ij} \in \Ext^1(\CC_{p_j},M_i)$  non-zero elements. Then $E$ is stable if $s > r(g-1)$.
 
 In particular for $s \geq r(g-1) +1$,
 \begin{equation} \label{eq5.13} 
B(r(g-1)+s,2r(g-1) +2s+2,rg+s) \neq \emptyset.  
 \end{equation}
\end{prop}

\begin{proof}
Suppose $G$ is a subbundle of $E$ contradicting stability, i.e.
$$
\frac{d_G}{n_G} \geq 2 + \frac{2}{r(g-1)+s}.
$$
Let $F := G \cap (D(K_C)^r \oplus M_1 \oplus \dots \oplus M_s)$. We have
$$
2n_F \geq d_{F} \geq d_G -2 > 2n_G -2=2n_F-2.
$$
Suppose first that $d_{F} = 2n_F$. Then $d_G = 2n_F + 1$ or $2n_F + 2$. So there exists an extension of one of the following forms
\begin{equation} \label{eq}
0 \ra F \ra G \ra \CC_{p_1} \oplus \CC_{p_2} \ra 0 
\end{equation}
or
$$
0 \ra F \ra G \ra \CC_{p_i} \ra 0
$$
with $i = 1$ or 2. Moreover, $F$ is a partial direct sum of factors of $D(K_C)^r \oplus M_1 \oplus  \cdots \oplus M_s$. This contradicts the generality assumptions on \eqref{e5.12}.

If $d_{F} = 2n_F -1$, then $d_G = 2n_F + 1$ and this 
implies that  $n_F \leq \frac{r(g-1)+s}{2}$. Moreover, there exists an extension of the form \eqref{eq}.  It follows from the semistability of  $D(K_C)^r \oplus M_1 \oplus \dots \oplus M_s$ that $F$ satisfies the hypotheses of Lemma \ref{lem}. 
The conclusion of the lemma contradicts the generality assumptions on \eqref{e5.12} provided $s > \frac{r(g-1)+s}{2}$,
i.e. $s \geq r(g-1) +1$.
\end{proof}

Let $C$ be a non-hyperelliptic curve of genus $g \geq 4$ and let $L$ be a line bundle on $C$ of degree $2g-3$ with $h^0(L) = g-1$. 
In fact, $L = K_C(-p)$ for some point $p \in C$ and it is generated. As in Section \ref{back}, define $E_L$ by the exact sequence
\begin{equation} \label{eq2.1}
0 \ra E_L^* \ra H^0(L) \otimes \cO_C \ra L \ra 0.
\end{equation}
Note that $E_L$ has rank $g-2$ and slope $2 + \frac{1}{g-2}$.

 \begin{lem} \label{l2.4}
 Let $C$ be a non-hyperelliptic curve of genus $g \geq 4$.
\begin{enumerate}
 \item $E_L$ is stable and $h^0(E_L) \leq \frac{3}{2}(g-2) + \frac{1}{2}$.
  \item Let $E$ be a bundle on $C$ of rank $n$ and degree $d$ such that $H^1(E\otimes L)=0$ and $h^0(E) > n + \frac{1}{g-1}(d-n)$. Then $h^0(E_L^* \otimes E)>0$.
\end{enumerate}\end{lem}

\begin{proof}
(1): Suppose $E_L$ has a proper quotient bundle $Q$ of slope $\leq 2$.
We can suppose that $Q$ is stable. Moreover $Q$ is generated. If $n_Q =1$, then the only possibility is $Q \simeq \cO_C$. This contradicts the fact that $h^0(E_L^*) =0$. 
If $n_Q \geq 2$, we have
$$
h^0(Q) \leq n_Q + \frac{1}{g}(d_Q - n_Q) \leq n_Q \left(1 + \frac{1}{g} \right)
$$
by Propositions \ref{p1.1} and \ref{p1.2}(i). Hence $h^0(Q) \leq n_Q$, which contradicts the fact that $Q$ is generated. Moreover,
$$
h^0(E_L) \leq \frac{3g-5}{2} = \frac{3}{2}(g-2) + \frac{1}{2}
$$ 
by Proposition \ref{p1.5}.

(2): Tensoring \eqref{eq2.1} with $E$ gives
\begin{equation} \label{equ}
0 \ra E_L^* \otimes E \ra E^{\oplus g-1} \ra E \otimes L \ra 0,
\end{equation}
Since $h^1(E \otimes L) =0$, we have $h^0(E \otimes L) = d + n(g-2)$ by Riemann-Roch and Serre duality. Hence
$$
(g-1)h^0(E) > n(g-1) +d-n = h^0(E \otimes L).
$$ 
It follows from \eqref{equ} that $h^0(E_L^* \otimes E) > 0$. 
\end{proof}

\begin{lem} \label{l2.5}
Let $C$ be a non-hyperelliptic curve of genus $g \geq 4$ and 
let $E$ be a semistable bundle on $C$ of rank $n$ and degree $d$ with $2 < \mu(E) < 2+ \frac{1}{g-3}$, such that $h^0(E) > n + \frac{1}{g-1}(d-n)$. 

Then $E_L$ can be embedded as a subbundle of $E$. In particular $\mu(E) \geq 2 + \frac{1}{g-2}$.
\end{lem}

\begin{proof}
Since $E\otimes L$ is semistable of slope $>2g-1$, $h^1(E\otimes L)=0$. So, by Lemma \ref{l2.4}, $h^0(E_L^* \otimes E) > 0$. Let $E_L \ra E$ be a non-zero homomorphism. 

If $E_L \ra E$ fails to be injective as a morphism of sheaves, then the stability of $E_L$ implies that $E$ has a subbundle $Q$ of rank $n_Q < g-2$ 
with $\mu(Q) > 2 + \frac{1}{g-2}$ and hence $\mu(Q) \geq 2 + \frac{1}{g-3}$. This contradicts the semistability of $E$.

If $E_L \ra E$ is injective as a morphism of sheaves, then the subbundle generated by $E_L$ has slope $\geq 2 + \frac{1}{g-2}$
with equality only if $E_L$ is a subbundle of $E$. Since $E$ is semistable, this implies that $E_L$ is a subbundle of $E$ of slope $2 + \frac{1}{g-2}$. 
\end{proof} 

\begin{ex} \label{ex3.7}
{\rm Let $C$ be a non-hyperelliptic curve, $r$ a positive integer and $n\ge rg$. By Proposition \ref{p1.2}, $B(n,2n,n+r)\ne\emptyset$. Let $F\in B(n,2n,n+r)$ and let $E$ be a bundle fitting into an exact sequence 
\begin{equation}\label{e3.9}
0 \ra F\ra E \ra \CC_q \ra 0
\end{equation}
with $q \in \CC$. It follows that $E\in B(n,2n+1,n+r)$. So $B(n,2n+1,n+r)\ne\emptyset$, giving a point
$$
(\mu, \lambda) = \left(2 + \frac1n, 1+\frac{r}n \right).
$$
in the BN-map. As $n \ra \infty$ for any fixed $r$, $(\mu, \lambda) \ra (2, 1)$.
}
\end{ex}

\begin{rem}\label{r3.10}
{\rm
For $r \leq g-1$, the BN-loci shown to be non-empty in Example \ref{ex3.7} are also shown to be non-empty in Proposition \ref{p3.3}.
On the other hand, Proposition \ref{p3.3} also includes examples not covered by Example \ref{ex3.7}. Of course, Example \ref{ex3.7} 
allows the possibility that $r\ge g$. In all cases one obtains subschemes of positive codimension in
the corresponding BN-loci.}
\end{rem}

\begin{rem}\label{r3.11}
{\rm In both Proposition \ref{p3.3} and Example \ref{ex3.7}, we have $h^0(F)>0$. This implies by Riemann-Roch that $h^1(F)>0$, so the general extension \eqref{e4.4} or \eqref{e3.9} yields a bundle $E$ with $h^0(E)=h^0(F)$. Hence $E$ is not generated. In fact, we have always $h^0(E)\ge n_E+r$ and, by Lemma \ref{l2.5}, $h^0(E)\le n_E+\frac1{g-1}(n_E+1)$. Hence, if $\frac1{g-1}(n_E+1)<r+1$, then $h^0(E)=n_E+r=h^0(F)$, implying that all $E$ given by extensions  \eqref{e4.4} or \eqref{e3.9} fail to be generated. The required condition is $s\le g-3$ for \eqref{e4.4} and $n-r(g-1)\le g-3$ for \eqref{e3.9}.

Suppose now that $r=1$. An alternative proof that $B(n,2n+1,n+1)\ne\emptyset$ for $n\ge g$ (in fact for $n\ge g-2$ and even $n\ge1$ when $g\le4$) can be given using \cite[Corollary 5.2]{bbn2}; this gives the result for general $C$ and it follows by semicontinuity for arbitrary $C$. For general $C$, the bundles constructed in this way are always generated. The condition of generality may not be needed here. It is certainly the case that, whenever $C$ is non-hyperelliptic, the bundles in $B(n,2n+1,n+1)$ which 
are elementary transformations of bundles in $B(n,2n,n+1)$ form a subscheme of positive codimension; this follows from \cite[Theorem 4.4(d)]{top3}.
}
\end{rem}

\section{Upper bounds for genus 4}\label{upper}

Let $C$ be a non-hyperelliptic curve of genus 4. In this section we will establish some upper bounds for the 
non-emptiness of $B(n,d,k)$ and $\widetilde B(n,d,k)$. Note that we already have such bounds for $0 \leq d < \frac{5n}{2}$ by Propositions \ref{p1.1}
and \ref{p1.2} and Lemma \ref{l2.5}.

\begin{lem} \label{l4.1}
{\em (i)} Let $E$ be a semistable bundle on $C$ of rank $n$ and slope $\frac{5}{2}$. Then 
$$
h^0(E) \leq \frac{3n}{2}.
$$
{\em (ii)} 
$$
B(2,5,3) = \widetilde B(2,5,3) = \{ E_L \; | \; L \simeq K_C(-p) \; \mbox{for some} \; p \in C \}.
$$
\end{lem}

\begin{proof}
(i): For $n = 2, h^0(E) \leq 3$ by Proposition \ref{p1.5}. For $n > 2$, suppose $h^0(E) > \frac{3n}{2}$. By Lemma \ref{l2.5}, $E_L$ can be embedded
as a subbundle of $E$. So $E$ is strictly semistable and $E/E_L$ is semistable of slope $\frac{5}{2}$. The result follows by induction.

(ii): Suppose $E \in B(2,5,3)$. If $E$ is not generated, then $E$ 
possesses a subsheaf $E'$ with $E/E' \ \simeq \CC_p$ for some point $p \in C$ and $h^0(E') = 3$. 
Now $E'$ is semistable of slope 2. This contradicts Proposition \ref{p1.2}. So we have an exact sequence
$$
0 \ra L^* \ra H^0(E) \otimes \cO_C \ra E \ra 0.
$$
Dualizing this sequence we obtain the result. 
\end{proof}

Recall that $C$ is trigonal and has either one or two trigonal bundles. Call these $T$ and $T'$, where
possibly $T$ is isomorphic to $T'$, and recall that $T \otimes T' \simeq K_C$.

\begin{lem} \label{l2.6}
Let $E$ be a stable bundle on $C$ of rank $n \geq 2$ with $\mu(E) = 3$. Then 
$$
h^0(E) \leq \frac{3n}{2}.
$$
\end{lem}

\begin{proof}
Suppose $h^0(E) > \frac{3n}{2}$. Consider the exact sequence
\begin{equation} \label{eq2.2}
0 \ra T^* \otimes E \ra H^0(T) \otimes E \ra T \otimes E \ra 0.
\end{equation}
Then $T^* \otimes E$ is stable of degree 0. So $h^0(T^* \otimes E) = 0$ and 
$$
h^0(T \otimes E) \geq 2h^0(E) > 3n.
$$
By Riemann-Roch and Serre duality, $h^0(T' \otimes E^*) > 0$. So there exists a non-zero homomorphism $E \ra T'$, which contradicts the stability of $E$.
\end{proof}

\begin{rem} \label{rem4.2}
{\rm There exist semistable bundles of rank $n \geq 2$ and degree $3n$ with $h^0 = 2n$, namely direct sums of copies of $T$ and $T'$. When $T \not \simeq T'$, 
these are the only bundles computing $\Cliff_n(C)$ (\cite[Theorem 3.2]{ln})}.
\end{rem}

\begin{prop} \label{p4.3}
Let $E$ be a semistable bundle on $C$ of rank $n$ and degree $d$ with  $\mu(E) \geq \frac{5}{2}$.
Then 
$$
h^0(E) \leq d-n.
$$
\end{prop}

\begin{proof}
The proof is by induction on $n$. For $n=1$ the result is trivial in view of Lemma \ref{l3.1}. For $n=2$, by Lemma \ref{l3.1} we can assume that  $d=5$. By 
Proposition \ref{p1.5}, $h^0(E) \leq \frac{7}{2}$. Since $n=2$, this implies that $h^0(E) \leq d-n$.

Now suppose $n \geq 3$ and the proposition is proved for rank $\leq n-1$.
Suppose 
$$
h^0(E) > d-n.
$$
By Lemma \ref{l3.1}, this implies that $\mu(E) < 3$.
Let $G$ be a proper quotient bundle of $E$ of minimal slope and consider the exact sequence
$$
0 \ra F \ra E \ra G \ra 0.
$$
We can suppose that $G$ is stable.
Moreover, it is immediate that $G$ satisfies the inductive hypothesis.

We claim that $F$ is semistable. If not, let $F'$ be a proper subbundle of $F$ with $\mu(F')>\mu(F)$. By semistability of $E$, we have $\mu(F')\le\mu(E/F')$. By minimality of $\mu(G)$, we have also $\mu(E/F')\ge\mu(G)$. Hence
\begin{eqnarray*}
d=n_F\mu(F)+(n-n_F)\mu(G)&<&n_F\mu(F')+(n-n_F)\mu(E/F')\\
&\le&n_{F'}\mu(F')+(n-n_{F'})\mu(E/F')=d,
\end{eqnarray*}
a contradiction.

From Lemma \ref{l2.5}, we see that $E_L$ can be embedded in $E$.
Then
$$
\mu(E/E_L) = \frac{d-5}{n-2} \geq \mu(G) = \frac{d-d_F}{n-n_F}.
$$
This is equivalent to
\begin{equation} \label{eq2.3}
(n-2)d_F \geq (d-5)n_F +5n -2d. 
\end{equation}
Now suppose $d_F < \frac{5}{2}n_F$. Substituting in \eqref{eq2.3} and simplifying, we obtain $n_F < 2$. So 
$$
\mbox{either} \quad \mu(F) \geq \frac{5}{2} \quad \mbox{or} \quad n_F = 1.
$$
In the first case $F$ satisfies the inductive hypothesis. In the second case we claim that $d_F = 2$.
Certainly $d_F \leq 2$. It is therefore sufficient to prove that $E$ has a line subbundle of degree 2. 
For this note that $L \simeq K_C(-p)$ for some point $p \in C$.
So $L \simeq T \otimes T'(-p)$. Hence there exists a non-zero homomorphism $T \ra L$ and therefore a non-zero homomorphism
$$
D(L) = E_L \ra D(T) = T.
$$
The kernel of this homomorphism is a line bundle of degree 2 and this embeds into $E$, since $E_L \subset E$.
Hence $d_F = 2$ and
$$
h^0(F) \leq 1 = d_F -n_F.
$$
The result now follows by induction.
\end{proof}

\begin{cor} \label{c4.4}
Suppose that $E$ is semistable of rank $n$ and degree $d$ with  $\frac{5}{2} < \mu(E) < 3$. Then  
$$
h^0(E) < d-n.
$$  
\end{cor}

\begin{proof}
The proof is by induction on $n$, the cases $n=1$ and $n=2$ being trivial. So assume $n \geq 3$ and 
$h^0(E) = d-n$. 

Following through the proof of the proposition, we have $h^0(F) \leq d_F - n_F$, using Proposition \ref{p4.3}, 
and $h^0(G) \leq d_G - n_G$, using the inductive hypothesis and Lemma \ref{l3.1}. So 
$$
h^0(F) = d_F - n_F \quad \mbox{and} \quad h^0(G) = d_G - n_G.
$$
Then, by Proposition \ref{p1.5} and Remark \ref{rem4.2} together with the inductive hypothesis, 
we see that $\mu(G)=3$ and $G = T$ or $T'$. 

If $n_F = 1$, then $d_F = 2$ and 
$\mu(E) = \frac{5}{2}$, a contradiction. So $\mu(F) \geq \frac{5}{2}$ and the only possibility for 
$h^0(F) = d_F - n_F$ is when $\mu(F) = \frac{5}{2}$. Now \eqref{eq2.3} implies that $n_F = 2$. So $d_F = 5$ and $h^0(F) = 3$.
Moreover, all sections 
of $G$ lift to $E$. This means that the map
\begin{equation} \label{eq4.3}
H^0(G) \otimes H^0(K_C \otimes F^*) \ra H^0(G \otimes K_C \otimes F^*)
\end{equation}
is not surjective. The kernel of this map is $H^0(G^* \otimes K_C \otimes F^*)$. 
The semistable bundle $G^* \otimes K_C \otimes F^*$ has rank 2 and slope $\frac{1}{2}$. So by Proposition \ref{p1.1}, $h^0(G^* \otimes K_C \otimes F^*) \leq 1$.
A dimensional calculation shows that \eqref{eq4.3} is surjective,  a contradiction.
\end{proof}

\begin{lem} \label{l2.8}
Let $C$ be a non-hyperelliptic curve of genus $4$ with $2$ distinct trigonal bundles $T$ and $T'$. 
Let $E$ be a semistable bundle of rank $n$ and degree $d$ with 
$$
\frac{5}{2} \leq \mu(E) < 3.
$$
Then 
$$
h^0(E) \leq \frac{d}{2} + \frac{n}{4}.
$$
\end{lem}

\begin{proof}
We proceed exactly as in the proof of Proposition \ref{p4.3} with the improved inequality. We have only to show that, 
if $h^0(E) > \frac{d}{2} + \frac{n}{4}$ and $n \geq 3$, then $\mu(G) < 3$.

Consider the exact sequence \eqref{eq2.2}. Since $T^* \otimes E$ is semistable of negative degree, we have $h^0(T^* \otimes E) = 0$ 
and hence 
$$
h^0(T \otimes E) > d + \frac{n}{2}.
$$
By Riemann-Roch and Serre duality, we obtain
$$
h^0(T' \otimes E^*) > \frac{n}{2}
$$
and similarly
$$
h^0(T \otimes E^*) > \frac{n}{2}.
$$
Hence we have a homomorphism 
$$
E \ra T^{\oplus \lceil \frac{n+1}{2} \rceil} \oplus T'^{\oplus \lceil \frac{n+1}{2} \rceil}
$$
such that $E$ does not map into any proper direct factor. If this map is injective as a morphism of sheaves,
then we have a generically surjective homomorphism
$$
\cO_C^{\oplus \lceil \frac{n+1}{2} \rceil} \oplus {(T \otimes T'^*)}^{\oplus \lceil \frac{n+1}{2} \rceil} \ra E^* \otimes T.
$$
From this we deduce an exact sequence
\begin{equation} \label{eq2.4}
0 \ra \cO_C^{\oplus r} \oplus ( T \otimes T'^*)^{\oplus n-r} \ra E^* \otimes T \ra \tau \ra 0 
\end{equation}
for some $r$ where $\tau$ is a torsion sheaf of length $= \deg(E^* \otimes T) \leq \frac{n}{2}$.

If we have equality, then $\mu(E) = \frac{5}{2}$ and $h^0(E) \leq \frac{3n}{2} = \frac{d}{2} + \frac{n}{4}$ by Lemma \ref{l4.1}.
We can therefore assume that $\tau$ is of length $< \frac{n}{2}$. The extensions \eqref{eq2.4} are classified by
$$
\bigoplus^r \Ext^1(\tau,\cO_C) \oplus \bigoplus^{n-r} \Ext^1(\tau,T \otimes T'^*).
$$
If $r \geq \frac{n}{2}$, then the components of the first factor must be linearly dependent, which implies that $\cO_C$ is a direct factor of $E^* \otimes T$.
Similarly, if $r< \frac{n}{2}$, then $T \otimes T'^*$ is a direct factor of $E^* \otimes T$. In either case $E$ fails to be semistable.

It follows that the morphism $E \ra T^{\oplus \lceil \frac{n+1}{2} \rceil} \oplus T'^{\oplus \lceil \frac{n+1}{2} \rceil}$ is not injective. Since $E$ does not map into any
proper direct factor, it follows that $E$ possesses a proper quotient bundle of slope $< 3$. This implies $\mu(G) < 3$ and hence $G$ satisfies the inductive hypothesis. 
This completes the proof.
\end{proof}

\begin{rem}
{\rm
A plausible conjecture for the range $2 < \mu \leq \frac{5}{2}$ would be $k \leq \frac{d}{2} + \frac{n}{4}$ (compare Lemma \ref{l2.8}).
However, taking $r=3$ and $s = 0$ in Proposition \ref{p3.3}, we see that $B(9,19,12) \neq \emptyset$. So the conjecture is not valid.
}
\end{rem}

We finish by combining the results of this section with some from previous sections to obtain the following theorem (see also the figure in Section \ref{map}).

\begin{theorem} \label{thm4.8}
Let $C$ be a non-hyperelliptic curve of genus $4$ with trigonal bundles $T, T'$. Suppose that $\widetilde B(n,d,k) \neq \emptyset$ where $n > 0, 0 \leq d \leq 3n$ and $k \geq 1$. 
Then one of the following holds.
\begin{enumerate}
 \item[(i)] $d=0, \; k \leq n$;
 \item[(ii)] $0 < d \leq 2n$ and either $k \leq n + \frac{1}{4}(d-n)$ or $d = 2n$ and $k \leq \frac{4n}{3}$;
 \item[(iii)] $2n < d \leq \frac{5}{2}n, \; k \leq n + \frac{1}{3}(d-n)$;
 \item[(iv)] $\frac{5}{2}n < d < 3n$ and $T \not \simeq T', \; k \leq \frac{d}{2} + \frac{n}{4}$;
 \item[(v)] $\frac{5}{2}n < d < 3n$ and $T\simeq T',\; k < d-n$;
 \item[(vi)] $d = 3n, \; k \leq 2n$.
\end{enumerate}
\end{theorem}

\begin{proof}
For (i) and (ii), see Propositions \ref{p1.1} and \ref{p1.2}. (iii) follows from Lemma \ref{l2.5}, (iv) is Lemma \ref{l2.8}, (v) is Corollary \ref{c4.4} and (vi) is immediate
from Proposition \ref{p1.5}.
\end{proof}

\begin{rem}
{\rm
In the following cases listed in Theorem \ref{thm4.8}, we can definitely state that $B(n,d,k) = \emptyset$.
\begin{enumerate}
 \item[(i)] $d=0, \; n \geq 2$;
 \item[(ii)] $d= n = k \geq 2$ and $d = 2n,\; k >\frac{5}{4}n, \; (n,d,k) \neq (3,6,4)$;
 \item[(vi)] $d = 3n,\; k > \frac{3}{2}n, (n,d,k) \neq (1,3,2)$.
\end{enumerate}

}
\end{rem}

\section{Existence results for genus 4}\label{exist}

\begin{prop} \label{p3.7}
Let $C$ be a non-hyperelliptic curve of genus $4$. Suppose $2n < d \leq 3n$.
\begin{enumerate}
 \item[(i)] If $k \leq \frac{n}{2} + \frac{d}{4}$ or $(n,d,k) = (3,9,4)$,  then $B(n,d,k) \neq \emptyset$;
 \item[(ii)] if $k \leq \frac{n}{2} + \frac{d}{4}$ or $d = 3n$ and $k \leq 2n$, then $\widetilde B(n,d,k) \neq \emptyset$;
 \item[(iii)] if $d < 3n, d-2n \equiv \ell' \mod 4$ with $1 \leq \ell' \leq 4$ and $k \leq \frac{d+ \ell'}{2} - 2$, then
 $B(n,d,k) \neq \emptyset.$
 \end{enumerate}
\end{prop}

\begin{proof}
(i) follows from Propositions \ref{p1.1} and \ref{p1.2} by tensoring by an effective line bundle 
of degree 1 (see also \cite[Theorem 4.1]{bmno}). (ii) follows from (i) and the existence 
of $T, T' \in B(1,3,2)$.

(iii): Write $d = 2n + 4\ell + \ell'$ with $1 \leq \ell' \leq 4$. Since $0 < 3n -d < n$, we have $B(n,3n-d,k') \neq \emptyset$ for $k' \leq n + \frac{1}{4}(2n-d)$ according 
to Proposition \ref{p1.1}. In particular 
$$
B(n,3n-d,n-\ell -1) \neq \emptyset.
$$
Tensoring by $T$ and using \cite[Lemma 3.1]{bmno}, it follows that
$$
B(n,6n-d,2n-2\ell -2) \neq \emptyset.
$$
Now Serre duality gives $B(n,d,\frac{d+\ell'}{2}-2) \neq \emptyset$.
\end{proof}

\begin{cor} \label{cor5.2}
Let $C$ be as in the proposition. Suppose that $2n+3 \leq d < 3n$. Then
$$
B(n,d,n+1) \neq \emptyset.
$$
\end{cor}

\begin{proof}
This follows from Proposition \ref{p3.7}(iii). 
\end{proof}

\begin{rem}
{\rm
In Proposition \ref{p3.7} (iii), the inequality on $k$ is equivalent to $k \leq \frac{d}{2}$ if $\ell' = 3$ or 4 and to $k \leq \frac{d}{2} -1$ if $\ell' = 1$ or 2.
From Proposition \ref{p3.3} and Example \ref{ex3.7}, we already have many examples in the region $2n < d < \frac{5}{2}n$ with $B(n,d,k) \neq \emptyset$ and $k > \frac{d}{2}$.
In these cases $l'=1$. Proposition \ref{p5.14} gives similar examples with $\ell'=2$.

In the following example, we construct further bundles $E \in B(n,d,k)$ with $\frac{5}{2}n < d < 3$ and $k > \frac{d}{2}$.
}
\end{rem}

\begin{ex}\label{ex5.2} {\rm 
Let $C$ be a non-hyperelliptic curve of genus 4 with $T \not  \simeq T'$ and let $F$ be a stable bundle of rank $\ell$ and degree $\delta+3\ell$ on $C$ with $\ell, \delta \geq 1$.
Then $T \otimes F^*$ is a stable bundle of negative degree. So $h^0(T \otimes F^*) = 0$ and $h^1(T \otimes F^*) = \delta + 3 \ell$.
Similarly $h^1(T' \otimes F^*) = \delta + 3 \ell$. Hence there exists a unique bundle $E$ fitting into the exact sequence
$$
0 \ra T^{\oplus \delta + 3 \ell} \oplus T'^{\oplus \delta + 3 \ell} \ra E \ra F \ra 0
$$
such that no factor of $T^{\oplus \delta + 3 \ell} \oplus T'^{\oplus \delta + 3 \ell}$ splits off $E$. Note that
$$
\mu(E) = \frac{7\delta + 21\ell}{2\delta + 7 \ell}  \quad \mbox{and} \quad h^0(E) \geq 4\delta + 12 \ell.
$$
If $E$ is not stable, we have a diagram 
\begin{equation} \label{equ5.1}
\xymatrix{
& 0 \ar[d] & 0 \ar[d] & 0 \ar[d] &\\
0  \ar[r] & K \ar[r] \ar[d] & G \ar[r] \ar[d] & G' \ar[r] \ar[d] & 0\\
0 \ar[r] & T^{\oplus \delta+3 \ell} \oplus T'^{\oplus \delta + 3 \ell} \ar[r] \ar[d] &  E \ar[r] \ar[d] & F \ar[r] \ar[d] & 0\\
0 \ar[r] & Q \ar[r] \ar[d] & R \ar[r] \ar[d] & S \ar[r] \ar[d] & 0\\
& 0 & 0 & 0 &
}
\end{equation}
with 
\begin{equation} \label{equ5.2}
\frac{d_{G'} + d_K}{n_{G'} + n_K} \geq \frac{7 \delta + 21 \ell}{2 \delta + 7 \ell}.
\end{equation}

Here $G'$ is a non-zero subsheaf of $F$ of slope $> 3$, $Q$ and $R$ are vector bundles and $S$ is a coherent sheaf.

Since $F$ is stable, $h^0(F^* \otimes T) = 0$, hence $\Hom(S,T) = 0$. From the right hand vertical exact sequence we obtain
$$
\dim \Ext^1(S,T) = -\chi(F^* \otimes T) + \chi(G'^* \otimes T) = \delta + 3 \ell - d_{G'}.
$$
The same holds for $\Ext^1(S,T')$. 

If $Q$ has $T^{\oplus \delta + 3 \ell -r} \oplus T'^{\oplus \delta + 3 \ell -s}$ as a direct factor with $\min \{r,s \} < d_{G'}$, then at least 
one factor $T$ or $T'$ splits off the bottom exact sequence. This gives a non-zero homomorphism $E \ra T$ or $E \ra T'$.
So either $T$ or $T'$ splits off $E$, a contradiction.
Hence, to prove the stability of $E$, it suffices to prove that, given diagram \eqref{equ5.1} with the stated conditions, then 
\begin{equation} \label{equ5.3}
\min\{ r,s\} < d_{G'}.
\end{equation}
}
\end{ex} 

This looks to be difficult to prove in all cases, but we have 

\begin{prop}
If $1 \leq \delta \leq 2$, then $E$ is stable. 
\end{prop}

\begin{proof}
Suppose first that $G' = F$ and $\mu(K) = 3$. Then $K$ maps onto a proper direct factor $T^{\oplus r} \oplus T'^{\oplus s}$ of 
$T^{\oplus \delta+3 \ell} \oplus T'^{\oplus \delta + 3 \ell}$. It follows that the middle horizontal sequence in \eqref{equ5.1} is obtained from the top one 
by taking a direct sum with $T^{\oplus \delta +3 \ell -r} \oplus T'^{\oplus \delta + 3 \ell -s}$, a contradiction. This completes the case $\delta = 1$, 
since otherwise $\mu(G) \leq 3$. 

For $\delta = 2$ we are left with the two possibilities, 
\begin{enumerate}
 \item[(a)] $G' = F$ and $d_K = 3n_K -1$,
 \item[(b)] $d_{G'} = 3n_{G'} + 1$ and $ d_K = 3n_K$.
\end{enumerate}

(a): From \eqref{equ5.2} we get 
$$
d_K + d_{G'} \geq \left( 3 + \frac{2}{4+7 \ell} \right) (n_K + n_{G'}).
$$
In other words,
$$
3n_K -1 + 3\ell +2 \geq \left( 3 + \frac{2}{4+7 \ell} \right) (n_K + \ell),
$$
which gives
$$
n_K \leq \frac{4+5\ell}{2} < 2 + 3 \ell.
$$
Now, if $K$ is semistable, then $K^* \otimes T$ is semistable of slope $\leq 1$, so 
\begin{equation} \label{equ5.4}
h^0(K^* \otimes T) \leq n_K
\end{equation}
by Proposition \ref{p1.1}.

If $K$ is not semistable, it must have a proper semistable subbundle $K'$ of slope 3 while the quotient $K/K'$ is also semistable. Both $K'^* \otimes T$ 
and $(K/K')^* \otimes T$ are semistable of slope $\leq 1$. So again \eqref{equ5.4} holds. It now follows that at least one factor $T$ (and one factor $T'$)
splits off $E$, a contradiction.

(b): In this case \eqref{equ5.2} yields
$$
n_K + n_{G'} \leq \frac{4 + 7 \ell}{2}.
$$
By stability of $F$,
$$
n_{G'} > \frac{\ell}{2}.
$$
So 
$$
n_K < 2 + 6 n_{G'} = 2d_{G'}.
$$
Now $K$ maps into $T^{\oplus \delta +3 \ell}  \oplus T'^{\oplus \delta + 3 \ell}$ as a direct factor $T^{\oplus r} \oplus T'^{\oplus s}$ with $r + s = n_K$. 
It follows at once that \eqref{equ5.3} holds. This completes the proof of the proposition.
\end{proof}

\begin{cor} \label{c3.9}
If $1 \leq \delta \leq 2$, we have
\begin{equation} \label{e5.5}
K_C \otimes E^* \in B(2\delta +7 \ell, 5\delta+ 21 \ell, 3\delta + 12 \ell).
\end{equation}
\end{cor}

\begin{rem}
{\rm
This is proved only when $T \not \simeq T'$. If $T \simeq T'$, then $C$ occurs in a flat family of 
curves whose general member is a non-hyperelliptic curve of genus 4 with $T \not \simeq T'$. 
It follows that $K_C \otimes E^* \in \widetilde B(2\delta +7 \ell, 5\delta+ 21 \ell, 3\delta + 12 \ell)$. If $\delta = 1$ or $\delta = 2$ and 
$\ell$ is odd, $\gcd(2\delta+7\ell, 5 \delta + 21 \ell) = 1$. So in these cases \eqref{e5.5}
still holds.
}
\end{rem}

\begin{rem}
{\rm
The bundles constructed in Corollary \ref{c3.9} all give rise to points $(\mu, \lambda)$ in the 
BN-map which lie on the line  $\lambda = \frac{3}{7}(\mu + 1)$. For $\frac{5}{2} <
\mu < 3$, this lies slightly under the line $\lambda = \frac{\mu}{2} + \frac{1}{4}$ from
Lemma \ref{l2.8}. In fact, the points $(\mu,\lambda)$ lie below the BN-curve.

In view of this and the other examples constructed above, the upper bounds of Section 4 in the 
case $T \not \simeq T'$ are close to being best possible, at least if we consider only piecewise
linear upper bounds. When $T \simeq T'$ on the other hand, we know of no stable bundles of rank 
$n \geq 2$ which are above the line $\lambda = \frac{\mu}{2} + \frac{1}{4}$ in the range $\frac52<\mu<3$, but the upper 
bound of Proposition \ref{p4.3} is $\lambda = \mu -1$.
}
\end{rem}
Finally, we construct some further examples of non-empty $B(n,d,k)$, using positive and negative elementary transformations of bundles
$E_{L_1} \oplus \cdots \oplus E_{L_r}$. We begin with a lemma.

\begin{lem} \label{l4.11}
Let $C$ be a non-hyperelliptic curve of genus $4$ and $L = K_C(-p)$ for some $p \in C$.
Then
\begin{enumerate}
 \item[(i)] $h^0(E_L^* \otimes T) = h^0(E_L^* \otimes T') = 1$;
 \item[(ii)] if $T \not \simeq T'$, the bundle $E_L$ possesses precisely $2$ line subbundles of degree $2$, which are isomorphic to $T(-p)$ and $T'(-p)$ respectively;
 \item[(iii)] if $T \simeq T'$, the bundle $E_L$ possesses only one line subbundle of degree $2$, which is isomorphic to $T(-p)$. 
\end{enumerate}
\end{lem}

\begin{proof}
Since $E_L^* \otimes T$ is a stable bundle of slope $\frac{1}{2}$, we have $h^0(E_L^* \otimes T) \leq 1$ by Proposition \ref{p1.1}. Now $h^0(T) = 2$, so 
$h^0(K_C \otimes T^*) = 2$ and $h^0(K_C(-p) \otimes T^*) \geq 1$. Thus there is a non-zero homomorphism $T \ra L$ and hence also a non-zero 
homomorphism $D(L) \ra D(T)$, i.e. $E_L \ra T$. So $h^0(E_L^* \otimes T) = 1$ and similarly $h^0(E_L^* \otimes T') = 1$.

By (i), there is a homomorphism $E_L \ra T$, necessarily surjective, and the kernel is a line subbundle isomorphic to $T'(-p)$. On the 
other hand, if $F$ is any line subbundle of degree 2, $E_L/F$ has degree 3 and $h^0 \geq 2$, so is isomorphic to $T$ or $T'$. This implies (ii) and (iii).
\end{proof}

\begin{ex} \label{ex4.12}
{\rm Consider non-trivial extensions 
\begin{equation} \label{e4.5}
0 \ra E_{L_1} \oplus \cdots \oplus E_{L_r} \ra E \ra \CC_q \ra 0
\end{equation}
with $r \geq 1$ and $q \in C$. We have $L_i \simeq K_C(-p_i)$ and we suppose $p_1, \dots,p_r$ are distinct points of $C$.
We propose to prove that the general extension \eqref{e4.5} gives rise to a stable bundle $E$. 

If $E$ is not stable, the extension \eqref{e4.5} is induced from an extension
\begin{equation} \label{e4.6}
0 \ra F \ra F' \ra \CC_q \ra 0 
\end{equation}
with $F$ a subbundle of $E_{L_1} \oplus \dots \oplus E_{L_r}$ of rank $n_F$ and degree $d_F$, 
such that 
$$
\frac{d_F + 1}{n_F} \geq \frac{d_E}{n_E} = \frac{5r+1}{2r} \quad \mbox{and} \quad \frac{d_F}{n_F} \leq \frac{5}{2},
$$
since $E_{L_1} \oplus \cdots \oplus E_{L_r}$ is semistable.

Since the $E_i$ are pairwise non-isomorphic, the only subbundles of $E_{L_1} \oplus \cdots \oplus E_{L_r}$ of slope $\frac{5}{2}$ are the partial 
direct sums. The general extension \eqref{e4.5} is not induced from any of these.

Now suppose that $\frac{d_F}{n_F} < \frac{5}{2}$ and write
$$
\frac{5}{2} - \frac{d_F}{n_F} = \frac{a}{2n_F}
$$
with $a$ a positive integer. Then we require 
$$
\frac{a}{2n_F} \leq \frac{1}{n_F} - \frac{1}{2r}, \quad \mbox{i.e.} \quad a \leq 2 - \frac{n_F}{r}.
$$
This gives $n_F \leq r$ and $a=1$ and hence
$$
d_F = \frac{5n_F -1}{2}.
$$
If $n_F = 1$, the subbundle $F$ projects into some factor $E_{L_i}$ isomorphically onto a line subbundle of $E_{L_i}$ of degree 2. By Lemma \ref{l4.11}
and the assumption that $p_1, \dots, p_r$ are all distinct, $F$ must be one of the one or two line subbundles of degree 2 of the factor $E_{L_i}$.
The general extension \eqref{e4.5} is not induced from either of the corresponding extensions \eqref{e4.6}. In particular, this completes the proof of the stability
of $E$ in the case $r=1$.

It remains to consider the case 
$$
n_F \; \mbox{odd}, \;\;\; 3 \leq n_F \leq r \;\;\; \mbox{and} \;\;\; d_F= \frac{5n_F-1}{2}.
$$
If the projection of $F$ into $E_{L_i}$ has rank 1, then the kernel has rank $n_F -1$ and degree $\geq \frac{5n_F -1}{2} - 2 = \frac{5}{2}(n_F -1)$.
Hence its slope must be $\frac{5}{2}$, which means that it must be a partial direct sum of factors of $E_{L_1} \oplus \cdots \oplus E_{L_r}$
not including $E_{L_i}$. Thus \eqref{e4.5} is induced from a partial direct sum of factors.

On the other hand, if $F$ projects into $E_{L_i}$ with rank 2, then the kernel has rank $n_F -2$ and degree $\frac{5(n_F-2)-1}{2}$. It follows by induction that $F$ is again 
contained in a partial direct sum of factors. This completes the proof that the general extension \eqref{e4.5} gives rise to a stable bundle $E$.

Hence for all $r \geq 1$ we have
$$
B(2r,5r+1,3r) \neq \emptyset.
$$
Note that, as $r \ra \infty$, the slope $\frac{5r+1}{2r} \ra \frac{5}{2}$. Also, if $r=1$, we obtain
\begin{equation}  \label{equ5.6}
B(2,6,3) \neq \emptyset,
\end{equation} 
thus showing that the upper bound of Lemma \ref{l2.6} can be attained.
}
\end{ex}

\begin{rem}
{\rm
By Proposition \ref{p3.7}, $B(n,3n,k) \neq \emptyset$ for $k \leq \frac{5}{4}n$.
We do not know whether there exist stable bundles with $\mu = 3$ and $\frac{5n}{4} < h^0 < \frac{3n}{2}$ except for the bundles 
$D(K_C)(p) \in B(3,9,4)$.
}
\end{rem}

\begin{ex} \label{ex5.10}
{\rm
Now consider negative elementary transformations of $E_{L_1} \oplus \cdots \oplus E_{L_r}$,
$$
0 \ra E \ra E_{L_1} \oplus \cdots \oplus E_{L_r} \ra \CC_q \ra 0
$$
with $L_1, \dots L_r$ as in Example \ref{ex4.12} and $q \in C$. Suppose that all the maps 
$E_{L_i} \ra \CC_q$ are non-zero. 
If $r = 1$, we suppose moreover that the restriction of the map $E_{L_1} \ra \CC_q$ 
to either of the line subbundles of $E_{L_1}$ of degree 2 is non-zero.
We propose again to prove that $E$ is stable.

If not, then there exists a proper subbundle $F$ of $E$ with
$$
\frac{d_F}{n_F} \geq \frac{5}{2} - \frac{1}{2r}.
$$
Then $F$ is a subbundle of $E_{L_1} \oplus \cdots \oplus E_{L_r}$. If $\frac{d_F}{n_F} = \frac{5}{2}$, then $F$ is a partial direct sum of factors of 
$E_{L_1} \oplus \cdots \oplus E_{L_r}$. This is impossible, since no $E_{L_i}$ is contained in $E$. Now write
\begin{equation} \label{eq5.9}
\frac{5}{2} - \frac{d_F}{n_F} = \frac{a}{2n_F}
\end{equation}
with $a$ a positive integer. This gives
$$
a \leq \frac{n_F}{r}.
$$
So $a = 1$ and $n_F \geq r$. Also 
\begin{equation} \label{eq5.10}
d_F = \frac{5n_F -1}{2}.
\end{equation}
If any of the projections $F \ra E_{L_i}$ has rank 1, then $\mu(F) \leq 2$ by stability of the $E_{L_i}$.
By \eqref{eq5.9} this means that $n_F = 1$ and hence  $d_F = 2, r = 1$. This is excluded by the definition 
of the exact sequence defining $E$.
 
So at least one of the projections $F \ra E_{L_i}$ has rank 2. By \eqref{eq5.10}, $n_F \geq 3$. Let $F'$ denote the kernel of $F \ra E_{L_i}$. Then 
$$
d_{F'} \geq \frac{5(n_F -2) -1}{2}
$$ 
with equality only if $F \ra E_{L_i}$ is surjective. If $F \ra E_{L_i}$ is not surjective, this contradicts the semistability of 
$E_{L_1} \oplus \dots \oplus E_{L_r}$
Since $n_{F'} = n_F -2$, the result follows by induction.

The conclusion is that 
\begin{equation} \label{e5.10}
B(2r,5r-1,3r-1) \neq \emptyset.
\end{equation}
}
\end{ex}

\begin{prop}   \label{p5.13}
Let $C$ be a non-hyperelliptic curve of genus $4$. Suppose that $2n < d < 3n$. Then for $n \geq 2$,
$$
B(n,d,n+1) \neq \emptyset.
$$
\end{prop}

\begin{proof}
If $d \geq 2n+3$, then $B(n,d,n+1) \neq \emptyset$ by Corollary \ref{cor5.2}. For $d \geq 2n+4$ this follows also from the fact that 
$B(n,d',n+1) \neq \emptyset$ for $n+4 \leq d' < 2n$ (see Proposition \ref{p1.1}). It remains to consider the cases $d = 2n+1$ and $d = 2n+2$.

For $n=2,\; d=5$, we have $E_L \in B(2,5,3)$. For $d = 2n+1$ with $n \geq 3$, the examples of Proposition \ref{p3.3} with $r=1$ show that $B(n,d,n+1) \neq \emptyset$.
Alternatively, for $n \geq 3, \; B(n,2n,n+1) \neq \emptyset$ by Proposition \ref{p1.2}. Taking an elementary transformation we obtain 
$B(n,2n+1,n+1) \neq \emptyset$. 

Now suppose $d = 2n+2$. For $n \geq 7$, we have $B(n,2n+2,n+1) \neq \emptyset$ by Proposition \ref{p5.14}. For $n = 3$, define $E$ by an exact sequence
\begin{equation}  \label{e5.11}
0 \ra E \ra D(K_C)(p) \ra \CC_p \ra 0
\end{equation}
for some $p \in C$.
Since $h^0(D(K_C)(p)) = h^0(D(K_C)) =4$ (see Lemma \ref{l2.6}), all sections of $D(K_C)(p)$ vanish at $p$. It follows that $h^0(E) = h^0(D(K_C)(p)) = 4$. Moreover, $E$ is stable, 
since 
$D(K_C)(p)$ is stable with integral slope. So $E \in B(3,8,4)$.

It follows from \cite[Proposition 7.6]{bbn1} that $B(4,10,5) \neq \emptyset$, at least if $C$ is Petri, i.e. $T \not \simeq T'$.
In fact, the proof of \cite[Proposition 7.6]{bbn1} assumes only that the dimensions of certain BN-loci are as expected. This is easy to see even when $T \simeq T'$.
For $n = 5$, $B(5,12,6) \neq \emptyset$ 
by \cite[Corollary 5.2]{bbn2} for $C$ general and hence for any $C$. For $n=6$, consider \eqref{e5.10} with $r=3$. 
We obtain $B(6,14,8) \neq \emptyset$. Hence also $B(6,14,7) \neq \emptyset$.
\end{proof}

\begin{cor} \label{c5.14}
Let $C$ be a non-hyperelliptic curve of genus $4$. Then for $n \geq 2$,
$$
B(n,d,n+1) \neq \emptyset \;\; \Leftrightarrow \;\; \beta(n,d,n+1) \geq 0.
$$
\end{cor}

\begin{proof}
Recall that by \eqref{equ2.1},
$$
\beta(n,d,n+1) = 3n^2 +1 -(n+1)(1 -d+4n).
$$
Hence 
$$
\beta(n,d,n+1) \geq 0 \;\; \Leftrightarrow \;\; d \geq 4+n - \frac{4}{n+1}.
$$
For $n \geq 4$, this gives $d \geq n+4$. If $d\le2n$, then $B(n,d,n+1) \neq \emptyset$ if and only if $d \geq n+4$ by Propositions \ref{p1.1} and \ref{p1.2}. 
The result now follows from the proposition. The cases $n = 2$ and $n = 3$ can be checked similarly.
\end{proof}

\begin{rem}
{\rm
We have shown above that $B(n,2n+2,n+1) \neq \emptyset$ for all $n$. If we could prove that this locus contains a generated bundle, then it would follow that the strong form
of Butler's conjecture \cite[Conjecture 1.3]{bbn2} holds in the case $g = 4,\; d=2n+2$. This is not currently known (see \cite[Section 6]{bbn2}).
}
\end{rem}

\section{Extremal bundles and bundles of low rank}

The bundles $D(K_C), E_L, T$ and $T'$ are extremal in the sense that the corresponding points of the BN-map (see Section \ref{map}) lie on the lines that we have established 
as the upper bounds for non-emptiness of $B(n,d,k)$. 

By Proposition \ref{p1.2}(i), $D(K_C)$ is the only stable bundle representing the point $(2,\frac{4}{3})$ and by Remark \ref{rem4.2}, $T$ and $T'$ are the only 
stable bundles representing the 
point $(3,2)$. According to Lemma \ref{l4.1}(ii), the only stable bundles of rank 2 representing the point $(\frac{5}{2},\frac{3}{2})$ are the bundles $E_L$ for $L = K_C(-p)$.
The following proposition shows that there are no stable bundles on the line $k = n + \frac{1}{3}(d-n)$ with $2n \leq d \leq \frac{5}{2}n$ except for the bundles  $D(K_C)$ and $E_L$.  

\begin{prop} \label{p6.1}
Let $C$ be a non-hyperelliptic curve of genus $4$. Suppose $k =n + \frac{1}{3}(d-n)$. Suppose further that $2 < \frac{d}{n} \leq \frac{5}{2}$ and, if $d = \frac{5}{2}n$,
then $n>2$.
Then 
$$
B(n,d,k) = \emptyset.
$$
\end{prop}

\begin{proof}
Suppose $E \in B(n,d,k)$. Note that $h^0(E) = k$ by Lemma \ref{l2.5}. We first claim that $E$ is generated. If not, then there is an exact sequence
$$
0 \ra F \ra E \ra \CC_p \ra 0
$$
with $h^0(F) = h^0(E) =k$. Let $L := K_C(-p)$. Since $E \otimes L$ is stable of slope $> 7$, it follows that $E \otimes L$ is generated. Hence 
$$
h^1(F \otimes L) = h^1(E \otimes L) = 0.
$$
It follows from Lemma \ref{l2.4}(2) that $h^0(E_L^* \otimes F) > 0$. 
This contradicts the stability of $E$. Hence $E$ is generated.

So we get an exact sequence
$$
0 \ra G^* \ra H^0(E) \otimes \cO_C \ra E \ra 0
$$
with $n_G = k-n,\; d_G = d$ and $h^0(G) \geq k$. It follows that $K_C \otimes G^*$ has rank $k-n$, degree $d-2n >0$ and $h^0 \geq n_G$. 
Any such bundle necessarily has a section with a zero. So $K_C \otimes G^*$ admits a line subbundle $M$ with $h^0(M) \geq 1$ and $d_M \geq 1$ and we get the diagram
$$
\xymatrix{
0 \ar[r] & E^* \ar[d]_{\alpha} \ar[r] & W \otimes \cO_C  \ar[d] \ar[r] & G \ar[r] \ar[d] & 0\\
0  \ar[r] & H^* \ar[r]  & V \otimes \cO_C  \ar[r] \ar[d] & K_C \otimes M^* \ar[r] \ar[d] & 0\\
&  & 0 & 0 &
}
$$
where $W$ is a subspace of $H^0(G)$ of dimension $k$ and $V$ is the image of $W$ in $H^0(K_C \otimes M^*)$.
Now $K_C \otimes M^*$ is not isomorphic to $\cO_C$, since $h^0(G^*) = 0$. Hence $\dim V \geq 2$ and $d_{K_C \otimes M^*} \geq 3$, since $C$ is non-hyperelliptic.

If $\alpha = 0$, then $E^*$ maps into $W' \otimes \cO_C$, where $W = W' \oplus V'$ and $V'$ maps isomorphically to $V$. It follows that $V' \otimes \cO_C$ 
maps to a trivial direct summand of $G$, contradicting the fact that $h^0(G^*) =0$. So $\alpha \neq 0$.

If $\dim V = 2$, then $\alpha(E^*)$ is a quotient line bundle of $E^*$ of degree $\leq -3$, contradicting the stability of $E$. Otherwise $L := K_C \otimes M^*$ 
is isomorphic to $K_C(-q)$ for some point $q$ and $H \simeq E_L$. In particular $H$ is stable of degree 5 and rank 2, again contradicting the stability of $E$.
 \end{proof}
 
 \begin{rem}
 {\rm
 By Corollary \ref{c4.4}, there are no further stable bundles representing points on the line segment joining $(\frac{5}{2},\frac{3}{2})$ to $(3,2)$ even when $T \simeq T'$. 
 }
 \end{rem}

We now turn to looking at bundles of rank $\le4$.

\begin{prop}
Suppose $0 <d < 12$. Then 
$B(2,d,k) \neq \emptyset$ if and only if $\beta(2,d,k) \geq 0$, except if $(d,k) = (2,2)$ or $(10,6)$. 
In the exceptional cases, $\beta(2,d,k) > 0$ and $\widetilde B(2,d,k) \neq \emptyset$, 
but $B(2,d,k) = \emptyset$. 
\end{prop}

\begin{proof}
Recall that $\beta(2,d,k) = 13 - k(k-d+6)$. So
$$
\beta(2,d,k) \geq 0 \;\; \Leftrightarrow \;\; d \geq k+6 - \frac{13}{k}.
$$
If $k = 1$, then $B(2,d,1)$  is non-empty if and only if $d > 0$ (see Proposition \ref{p1.1}). If $k = 2$, $\beta(2,d,2) \geq 0$ if and only if $d \geq 2$ and $\widetilde B(2,2,2) \neq \emptyset$, 
but $B(2,2,2) = \emptyset$ (again see Proposition \ref{p1.1}). For $k = 3$, the result follows from Corollary \ref{c5.14}. The rest follows from Serre duality and Riemann-Roch.
\end{proof}

\begin{prop}
Suppose $0 < d < 18$. Then 
$B(3,d,k) \neq \emptyset$ if and only if $\beta(3,d,k) \geq 0$, except if $(d,k) = (3,3), (15,9)$ or $(9,5)$. 
In the exceptional cases, $\beta(3,d,k) > 0$ and $\widetilde B(3,d,k) \neq \emptyset$, but $B(3,d,k) = \emptyset$.
\end{prop}

\begin{proof}
For $k \leq 3$, this follows from Proposition \ref{p1.1}. For $k = 4$, it follows from Corollary \ref{c5.14}.

Now suppose $k = 5$. We have $\beta(3,d,5) = 5d -42$. So $\beta(3,d,5) \geq 0$ if and only if $d \geq 9$. If $d \leq 8$, then $B(3,d,5) = \emptyset$ (see the BN-map
and Corollary \ref{c4.4}). 
Moreover, $B(3,9,5) = \emptyset$ by Lemma \ref{l2.6}, but $\widetilde B(3,9,5) \neq \emptyset$, since $B(1,3,2)$ and $B(2,6,3)$ are both non-empty. 
The rest follows from Serre duality and Riemann-Roch.
\end{proof}

\begin{prop}
Suppose $0 < d < 24$. Then 
$B(4,d,k) \neq \emptyset$ if and only if $\beta(4,d,k) \geq 0$, except if $(d,k) = (4,4), (20,12), (10,6), (12,7), (14,8)$ and possibly $(12,6)$. 
In the exceptional cases, $\beta(4,d,k) \ge 0$ and $\widetilde B(4,d,k) \neq \emptyset$, but $B(4,d,k) = \emptyset$ except possibly when $(d,k) = (12,6)$.
\end{prop}

\begin{proof}
For $k \leq 4$, this follows from Proposition \ref{p1.1}. For $k = 5$, it follows from Corollary \ref{c5.14}. 

For $k = 6$, we have $\beta (4,d,6) = 6d -59 \geq 0$
if and only if $d \geq 10$. In fact, $B(4,10,6) = \emptyset$ by Proposition \ref{p6.1} and $B(4,11,6) \neq \emptyset$ by Example \ref{ex4.12}.
It is easy to check that $\widetilde B(4,10,6)$ and $\widetilde B(4,12,6)$ are non-empty.

For $k = 7$, we have $\beta(4,d,7) = 7d - 84 \geq 0$ if and only if $d \geq 12$. 
If $d \leq 11$, then $B(4,d,7) = \emptyset$ (see the BN-map and Corollary \ref{c4.4}). Moreover, $B(4,12,7) = \emptyset$ by Lemma \ref{l2.6}.
It is easy to check that $\widetilde B(4,12,7) \neq \emptyset$.
The rest follows from Serre duality and Riemann-Roch.
\end{proof}

\section{BN-map for genus 4}\label{map}

The following figure is the most significant part of the BN-map for a non-hyperelliptic curve of genus 
4.

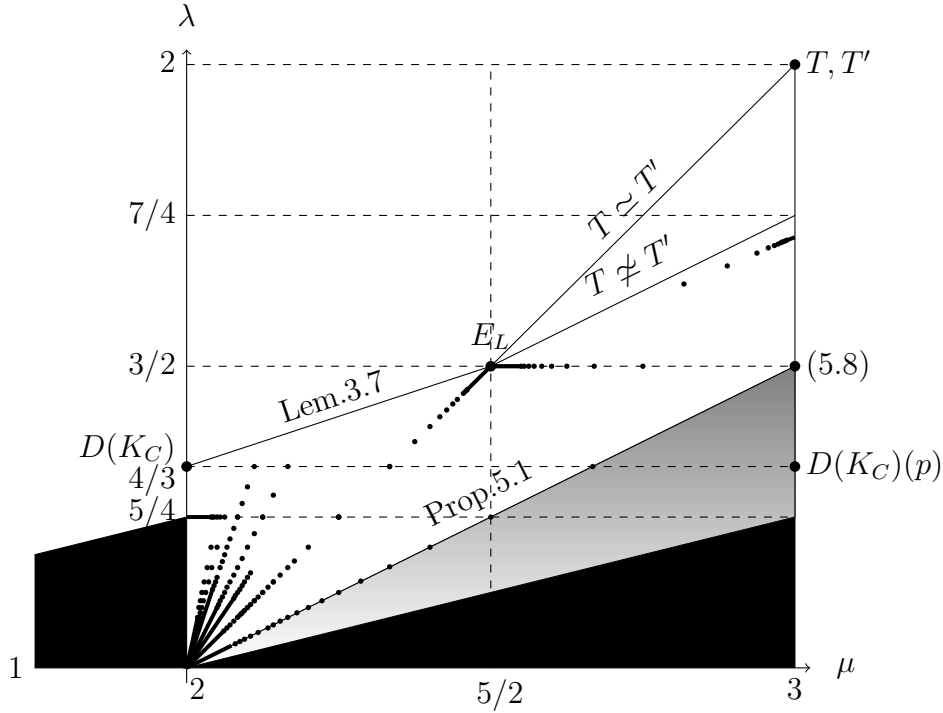
\begin{tikzpicture}
\path[draw][->] (2,-0.2) -- (2,8.2) node[pos=1.02,above] {$\lambda$} node[pos=.5,left] {$3/2$} node[pos=.32,left] {$4/3$} node[pos=0.98,left] {$2$} node[pos=0.27,left] {$5/4$}
                node[pos=.37,left] {$D(K_C)$};
\path[draw][->] (0,0 ) -- (10.2,0) node[pos=1.02,right] {$\mu$} node[pos=.6,below] {$5/2$} node[pos=0.21,below] {$2$} node[pos=0.98,below] {$3$} node[pos=0,left] {$1$}; 
               
\path[draw] (10, 0) -- (10,8) node[pos=1,right] {$T,T'$}; 
\path[draw] (2, 0) -- (10,2);
\path[draw][dashed] (2,8) -- (10,8);
\path[draw][dashed] (2,6) -- (10,6)  node[pos=0,left] {$7/4$};
\path[draw] (2,2.67) -- (6,4)  node[pos=0.5, sloped,above] {Lem.3.7};
\path[draw] (2, 0) -- (10,4)   node[pos=0.5, sloped,above] {Prop.5.1};
\path[draw] (6,4) -- (10,8) node[pos=0.1,left] {$E_L$} node[pos=0.5, sloped,above] {$T \simeq T'$};
\path[draw] (6,4) -- (10,6) node[pos=0.5,sloped,above] {$T \not \simeq T'$}; 
\path[draw] (0,1.5) -- (2,2); 
\path[fill=black] (2,0)--(10,0)--(10,2)--cycle;
\path[fill=black] (0,0)--(2,0)--(2,2)--(0,1.5)-- cycle;
\path[shade,draw] (2,0)--(10,2)--(10,4)-- cycle;
\fill (10,2.67) circle (2pt);
\fill (10,8) circle (2pt);
\fill (6,4) circle (2pt);
\fill (2,2.67) circle (2pt);
\fill (10,4) circle (2pt);
\path[draw][dashed] (6,0) -- (6,8);
\path[draw][dashed] (2,4) -- (10,4) node[pos=1,right]{$(5.8)$};
\path[draw][dashed] (2,2.67) -- (10,2.67) node[pos=1,right]{$D(K_C)(p)$};
\path[draw][dashed] (2,2) -- (10,2);

\fill (4,2) circle (1pt);
\fill (3,2) circle (1pt);
\fill (2.67,2) circle (1pt);
\fill (2.5,2) circle (1pt);
\fill (2.44,2) circle (1pt);
\fill (2.4,2) circle (1pt);
\fill (2.36,2) circle (1pt);
\fill (2.34,2) circle (1pt);
\fill (2.32,2) circle (1pt);
\path[draw,line width=1.5pt] (2.32,2) -- (2,2);

\fill (6,2) circle (1pt);
\fill (7.34,2.67) circle (1pt);
\fill (5.2,1.6) circle (1pt);
\fill (4.67,1.34) circle (1pt);
\fill (4.29,1.15) circle (1pt);
\fill (3.78,0.89) circle (1pt);
\fill (3.45,0.73) circle (1pt);
\fill (3.23,0.62) circle (1pt);
\fill (3.07,0.53) circle (1pt);
\fill (2.94,0.47) circle (1pt);
\fill (2.84,0.42) circle (1pt);
\fill (2.76,0.38) circle (1pt);
\fill (2.7,0.35) circle (1pt);
\fill (2.64,0.32) circle (1pt);
\fill (4,1) circle (1pt);
\fill (3.6,0.8) circle (1pt);
\fill (3.33,0.67) circle (1pt);
\fill (3.14,0.57) circle (1pt);

\path[draw,line width=1.5pt] (2,0) -- (2.59,0.295);

\fill (8,4) circle (1pt);
\fill (7.36,4) circle (1pt);
\fill (7,4) circle (1pt);
\fill (6.8,4) circle (1pt);
\fill (6.67,4) circle (1pt);
\fill (6.57,4) circle (1pt);
\fill (6.5,4) circle (1pt);
\fill (6.44,4) circle (1pt);
\fill (6.4,4) circle (1pt);
\path[draw,line width=1.6pt] (6,4) -- (6.4,4);

\fill (4,2) circle (1pt);
\fill (3.6,1.6) circle (1pt);
\fill (3.33,1.33) circle (1pt);
\fill (3.14,1.14) circle (1pt);
\fill (3,1) circle (1pt);
\fill (2.89,0.89) circle (1pt);
\fill (2.8,0.8) circle (1pt);
\fill (2.72,0.72) circle (1pt);
\fill (2.67,0.67) circle (1pt);
\fill (2.62,0.62) circle (1pt);
\fill (2.57,0.57) circle (1pt);
\fill (2.53,0.53) circle (1pt);
\fill (2.49,0.49) circle (1pt);
\path[draw,line width=1.5pt] (2.49,0.49) -- (2,0);

\fill (4.67,2.67) circle (1pt);
\fill (3.33,2.67) circle (1pt);
\fill (2.89,2.67) circle (1pt);

\fill (8.54,5.09) circle (1pt);
\fill (9.11,5.33) circle (1pt);
\fill (9.5,5.5) circle (1pt);
\fill (9.65,5.57) circle (1pt);
\fill (9.73,5.6) circle (1pt);
\fill (9.78,5.62) circle (1pt);
\fill (9.82,5.64) circle (1pt);
\fill (9.84,5.65) circle (1pt);
\fill (9.87,5.66) circle (1pt);
\fill (9.89,5.67) circle (1pt);
\path[draw,line width=1.5pt] (9.89,5.67) -- (10,5.71);

\fill (4,2) circle (1pt);
\fill (4.67,2.67) circle (1pt);
\fill (5,3) circle (1pt);
\fill (5.2,3.2) circle (1pt);
\fill (5.33,3.33) circle (1pt);
\fill (5.42,3.42) circle (1pt);
\fill (5.5,3.5) circle (1pt);
\fill (5.56,3.56) circle (1pt);
\fill (5.6,3.6,5.67) circle (1pt);
\fill (5.64,3.64) circle (1pt);
\fill (5.67,3.67) circle (1pt);
\fill (5.69,3.69) circle (1pt);
\path[draw,line width=1.5pt] (5.69,3.69) -- (6,4);

\fill (2.8,2.4) circle (1pt);
\fill (2.73,2.18) circle (1pt);
\fill (2.67,2) circle (1pt);
\fill (2.62,1.85) circle (1pt);
\fill (2.57,1.71) circle (1pt);
\fill (2.53,1.6) circle (1pt);
\fill (2.5,1.5) circle (1pt);
\fill (2.47,1.41) circle (1pt);
\fill (2.44,1.33) circle (1pt);
\fill (2.4,1.2) circle (1pt);
\fill (2.38,1.14) circle (1pt);
\fill (2.36,1.09) circle (1pt);
\fill (2.34,1.04) circle (1pt);
\path[draw,line width=1.5pt] (2.34,1.04) -- (2,0);

\fill (3.14,2.29) circle (1pt);
\fill (3,2) circle (1pt);
\fill (2.89,1.78) circle (1pt);
\fill (2.8,1.6) circle (1pt);
\fill (2.73,1.45) circle(1pt);
\fill (2.67,1.33) circle (1pt);
\fill (2.62,1.23) circle (1pt);
\fill (2.57,1.14) circle (1pt);
\fill (2.53,1.07) circle (1pt);
\fill (2.5,1) circle (1pt);
\fill (2.47,0.94) circle (1pt);
\fill (2.44,0.89) circle (1pt);
\fill (2.42,0.84) circle (1pt);
\fill (2.4,0.8) circle (1pt);
\path[draw,line width=1.5pt] (2.4,0.8) -- (2,0);

\fill (2.4,1.6) circle (1pt);
\fill (2.33,1.33) circle (1pt);
\fill (2.29,1.14) circle (1pt);
\fill (2.25,1) circle (1pt);
\fill (2.22,0.89) circle(1pt);
\fill (2.2,0.8) circle (1pt);
\fill (2.18,0.73) circle (1pt);
\fill (2.17,0.67) circle (1pt);
\path[draw,line width=1.5pt] (2.17,0.67) -- (2,0);

\fill (2.32,1.6) circle (1pt);
\fill (2.27,1.33) circle (1pt);
\fill (2.23,1.14) circle (1pt);
\fill (2.2,1) circle (1pt);
\fill (2.18,0.89) circle(1pt);
\fill (2.16,0.8) circle (1pt);
\fill (2.18,0.73) circle (1pt);
\fill (2.145,0.67) circle (1pt);
\path[draw,line width=1.5pt] (2.145,0.67) -- (2,0);

\fill (3.23,1.23) circle (1pt);
\fill (3.14,1.14) circle (1pt);
\fill (3.07,1.07) circle (1pt);
\fill (3,1) circle (1pt);
\fill (2.94,0.94) circle (1pt);
\fill (2.89,0.89) circle (1pt);
\fill (2.84,0.84) circle (1pt);
\fill (2.8,0.8) circle (1pt);
\fill (2.76,0.76) circle (1pt);
\fill (2.73,0.73) circle (1pt);
\fill (2.7,0.7) circle (1pt);

\fill (2.84,1.26) circle (1pt);
\fill (2.8,1.2) circle (1pt);
\fill (2.76,1.14) circle (1pt);
\fill (2.73,1.09) circle (1pt);
\fill (2.7,1.04) circle (1pt);
\fill (2.67,1) circle (1pt);
\fill (2.64,0.96) circle (1pt);
\fill (2.62,0.92) circle (1pt);
\path[draw,line width=1.5pt] (2.62,0.92) -- (2,0);

\end{tikzpicture}

The map plots $\lambda = \frac{k}{n}$ against $\mu = \frac{d}{n}$. The solid lines indicate 
the upper bounds for non-emptiness, given by Theorem \ref{thm4.8}. The shaded areas consist of 
points $(\mu, \lambda)$ for which there exists $(n,d,k)$ with 
$$
\frac{d}{n} = \mu,\quad  \frac{k}{n} =
 \lambda \quad \mbox{and} \quad B(n,d,k) \neq \emptyset.
 $$
 The black areas are given by Propositions \ref{p1.1}, \ref{p1.2} and  \ref{p3.7} (i) 
and all $B(n,d,k)$ corresponding to points in these areas are non-empty. In the grey area, which 
corresponds to Proposition \ref{p3.7} (iii), there are some $(n,d,k)$ for which possibly 
$B(n,d,k) = \emptyset$. 
However, for any $(\mu,\lambda)$ in this area there exist $(n,d,k)$ with $\mu = \frac{d}{n},\; \lambda = \frac{k}{n}$ such that $B(n,d,k) \neq \emptyset$.

The dots represent points for which some $B(n,d,k) \neq \emptyset$ 
and arise from \eqref{eq3.3}, \eqref{eq5.13}, Example \ref{ex3.7}, \eqref{e5.5}, Example \ref{ex4.12}, Example \ref{ex5.10} and Proposition \ref{p5.13}. Only a selection of the points is represented.

The BN-curve given by $\lambda(\lambda - \mu  + 3) = 3$ (or $\beta(n,d,k) = 1$) passes through the points $(2, \frac{-1 +\sqrt{13}}{2}),
( \frac{5}{2},\frac{3}{2})$ and $(3, \sqrt{3})$ and lies slightly
below the upper bound lines for the case $T \not \simeq T'$. We did not include it in the figure, 
because it is so close to these lines.
All the bundles constructed in this paper have $\beta(n,d,k) \geq 0$, but this does not rule out the possibility that $B(n,d,k)$ could be non-empty for some 
$(n,d,k)$ with $\beta(n,d,k) < 0$.

\end{document}